\documentclass[a4paper]{jpconf} 
\usepackage{iopams,amsgen,amsfonts,amssymb,amsbsy}
\input cyracc.def

\newtheorem{Lem}{Lemma}[section] 
\newtheorem{Prop}[Lem]{Proposition} 
\newtheorem{Thm}[Lem]{Theorem} 
\newtheorem{Cor}[Lem]{Corollary} 
\newtheorem{Def}[Lem]{Definition}

\newenvironment{proof}{{\sc Proof.}}{$\Box$}

\begin{document} 

\title{Symmetry classes connected with the magnetic Heisenberg ring} 
\author{Bernd Fiedler}
%\address{Bernd Fiedler \\ Mathematisches Institut \\ Universit\"at Leipzig\\ 
%Augustusplatz 10/11 \\ D-04109 Leipzig \\ Germany}
\address{Eichelbaumstr. 13, D-04249 Leipzig, Germany. URL: http://www.fiemath.de/}  
\ead{bfiedler@fiemath.de}  
%\subjclass{53B20, 15A72, 05E10, 16D60, 05-04} 

\begin{abstract}
It is well-known (see \cite{weyl4}) that for a Heisenberg magnet symmetry operators and symmetry classes can be defined in a very similar way as for tensors (see e.g. \cite{boerner2,weyl1,fie18}). Newer papers which consider the action of permutations on the Hilbert space $\mathcal{H}$ of the Heisenberg magnet are \cite{jak2lu,jaklu,jak2lu2,jak2lu2b}.

We define symmetry classes and commutation symmetries in the Hilbert space $\mathcal{H}$ of the {\rm 1D} spin-{\rm 1/2} Heisenberg magnetic ring with $N$ sites and investigate them by means of tools from the representation theory of symmetric groups $\mathcal{S}_N$ such as decompositions of ideals of the group ring $\mathbb{C}[\mathcal{S}_N]$, idempotents of $\mathbb{C}[\mathcal{S}_N]$, discrete Fourier transforms of $\mathcal{S}_N$, Littlewood-Richardson products. In particular, we determine smallest symmetry classes and stability subgroups of both single eigenvectors $v$ and subspaces $U$ of eigenvectors of the Hamiltonian $H$ of the magnet. Expectedly, the symmetry classes defined by stability subgroups of $v$ or $U$ are bigger than the corresponding smallest symmetry classes of $v$ or $U$, respectively. The determination of the smallest symmetry class for $U$ bases on an algorithm which calculates explicitely a generating idempotent for a non-direct sum of right ideals of $\mathbb{C}[\mathcal{S}_N]$.

Let $U_{\mu}^{(r_1,r_2)}$ be a subspace of eigenvectors of a a fixed eigenvalue $\mu$ of $H$ with weight $(r_1,r_2)$. If one determines the smallest symmetry class for every $v\in U_{\mu}^{(r_1,r_2)}$ then one can observe jumps of the symmetry behaviour. For ''generic'' $v\in U_{\mu}^{(r_1,r_2)}$ all smallest symmetry classes have the same maximal dimension $d$ and ''structure''. But $U_{\mu}^{(r_1,r_2)}$ can contain linear subspaces on which the dimension of the smallest symmetry class of $v$ jumps to a value smaller than $d$. Then the stability subgroup of $v$ can increase. We can calculate such jumps explicitely.

In our investigations we use computer calculations by means of the {\sf Mathematica} packages {\sf PERMS} and {\sf HRing}.
\end{abstract}

\section{The model of the magnetic Heisenberg ring} \label{sec2}
We summarize essential concepts of the one-dimensional (1D) spin-{1/2} Heisenberg model of a magnetic ring (see e.g. \cite{KarMu,KarHuMu,KarHuMu2}).

We denote by $\widehat{N}$ the set $\widehat{N}:=\{1,\ldots,N\}$ of the integers $1, 2,\ldots,N$ and by $\widehat{K}^{\widehat{N}}$ the set of all functions $\sigma : \widehat{N}\rightarrow\widehat{K}$.
\begin{Def}We assign to every function $\sigma\in\widehat{2}^{\widehat{N}}$ the sequence $|\sigma\rangle := |\sigma(1),\sigma(2),\ldots,\sigma(N)\rangle$ of its values over the set $\widehat{N}$. Then the Hilbert space of the {\rm 1D} spin-{\rm 1/2} Heisenberg ring with $N$ sites is the set of all formal complex linear combinations
$\mathcal{H} := \mathcal{L}_{\mathbb{C}}\left\{\;|\sigma\rangle\;|\;\sigma\in\widehat{2}^{\widehat{N}}\;\right\}$ of the $|\sigma\rangle$ in which the set $\mathcal{B} := \left\{\;|\sigma\rangle\;|\;\sigma\in\widehat{2}^{\widehat{N}}\;\right\}$ of all $|\sigma\rangle$ is considered a set of linearly independent elements. We equip $\mathcal{H}$ with the scalar product
\begin{equation}
\langle\sigma |\sigma'\rangle := {\delta}_{\sigma ,\sigma'} \;\;\;\mbox{and}\;\;\;
\langle u | v\rangle := \sum_{\sigma\in\widehat{2}^{\widehat{N}}} u_{\sigma} \overline{v}_{\sigma}
\;\;\;\mbox{for}\;\;\;
u = \sum_{\sigma\in\widehat{2}^{\widehat{N}}} u_{\sigma} |\sigma\rangle
\;,\;
v = \sum_{\sigma\in\widehat{2}^{\widehat{N}}} v_{\sigma} |\sigma\rangle
\in\mathcal{H}\,.
\end{equation}
\end{Def}
Obviously, $\mathcal{H}$ is an Hilbert space of dimension
$\dim\mathcal{H} = 2^N$, in which $\mathcal{B}$ is an orthonormal basis. The states $|\sigma\rangle\in\cal{H}$ describe {\it magnetic configurations} of the Heissenberg ring. $\sigma(k) = 2$ represents an {\it up spin} $\sigma(k) = \uparrow$ and $\sigma(k) = 1$ a {\it down spin} $\sigma(k) = \downarrow$ at site $k$.
\begin{Def}
According to {\rm\cite{KarMu,KarHuMu}}, the Hamiltonians $H_F$ {\rm (}$H_A${\rm )} of a {\rm 1D} spin-$\frac{1}{2}$ Heisenberg ferromagnet {\rm (}antiferromagnet{\rm )} of $N$ sites with periodic boundary conditions $S_{N+1}^{\alpha} := S_1^{\alpha}$, $\alpha\in\{+ , - , z\}$, are defined by
\begin{equation}
H_F := - J \sum_{k = 1}^N \left[ \frac{1}{2}\left(S_k^{+}S_{k+1}^{-} + S_k^{-}S_{k+1}^{+}\right) + S_k^z S_{k+1}^z \right]\,,\hspace*{0.5cm}H_A := - H_F\,,\hspace*{0.5cm}J = \mathrm{const.} > 0
\end{equation}
where the linear operators $S_k^{+}$, $S_k^{-}$ {\rm (}spin flip operators{\rm )} and $S_k^z$ are given by their values on the basis vectors $|\sigma\rangle$,
\begin{equation}
\begin{array}{l|rr}
 & _{k\downarrow}\hspace*{19pt} & _{k\downarrow}\hspace*{19pt} \\
 & \hspace*{0.5cm}|\ldots 1 \ldots\rangle & \hspace*{0.5cm}|\ldots 2 \ldots\rangle \\
\hline
S_k^{+} & |\ldots 2 \ldots\rangle & 0\hspace*{19pt} \\
S_k^{-} & 0\hspace*{19pt} & |\ldots 1 \ldots\rangle \\
S_k^z & - \frac{1}{2} |\ldots 1 \ldots\rangle & \frac{1}{2} |\ldots 2 \ldots\rangle \\
\end{array}
\end{equation}
\end{Def}
We consider only small rings with about $N\in\{4,\ldots,12\}$. We calculate eigenvalues and eigenvectors of the above Hamiltonians $H$ by expressing $H$ as a real, symmetric $2^N\times 2^N$-matrix $(H_{\mu \nu})$, defined by $H|\sigma\rangle = \sum_{\nu} H_{\sigma \nu} |\nu\rangle$, and using numerical standard diagonalization algorithms.

\section{Symmetry operators, symmetry classes and commutation symmetries} \label{sec3}%
Now we define symmetry classes in $\mathcal{H}$ in analogy to the definition of symmetry classes of tensors (see \cite{weyl4,boerner2,weyl1,fie18}).
\begin{Def}
Let $\mathcal{H}$ be the Hilbert space of a Heisenberg ring of $N$ sites, $\mathcal{S}_N$ be the symmetric group of the permutations of $1, 2,\ldots,N$ and $\mathbb{C}[\mathcal{S}_N]$ be the complex group ring of $\mathcal{S}_N$.
\begin{itemize}
\item[\rm (i)]{Following {\rm\cite{jak2lu2,jak2lu2b}} we define the action of a permutation $p\in\mathcal{S}_N$ on a basis vector $|\sigma\rangle\in\mathcal{H}$ by $p |\sigma\rangle := |\sigma\circ p^{-1}\rangle$.}
\item[\rm (ii)]{Every group ring element $a = \sum_p a_p p\in\mathbb{C}[\mathcal{S}_N]$ acts as so-called symmetry operator on the vectors $w = \sum_{\sigma} w_{\sigma} |\sigma\rangle\in\mathcal{H}$ by
\begin{equation}
aw :=
\sum_{p\in\mathcal{S}_N} \sum_{\sigma\in{\widehat{2}}^{\widehat{N}}} a_p w_{\sigma} p |\sigma\rangle =
\sum_{p\in\mathcal{S}_N} \sum_{\sigma\in{\widehat{2}}^{\widehat{N}}} a_p w_{\sigma} |\sigma\circ p^{-1}\rangle
\end{equation}
}
\end{itemize}
\end{Def}
A symmetry operator is a linear mapping $a:\mathcal{H}\rightarrow\mathcal{H}$ which possesses the following properties:
\begin{Prop}
\begin{itemize}
\item[\rm (i)]{It holds $p(q|\sigma\rangle) = (p\circ q)|\sigma\rangle$ for all $p,q\in\mathcal{S}_N$, where ''$\circ$'' is defined by $(p\circ q)(i) := p(q(i))$.}
\item[\rm (ii)]{It holds $a(bw) = (a\cdot b)w$ for all $a,b\in\mathbb{C}[\mathcal{S}_N]$ and $w\in\mathcal{H}$, where ''$\cdot$'' denotes the multiplication $a\cdot b := \sum_p \sum_q a_p b_q p\circ q$ of group ring elements.}
\item[\rm (iii)]{It holds $\langle au | v\rangle = \langle u |\overline{a}^{\ast}v\rangle$ for all $a\in\mathbb{C}[\mathcal{S}_N]$ and $u,v\in\mathcal{H}$, where $\overline{a}$ denotes the complex conjugate of $a$ and the element $a^{\ast}\in\mathbb{C}[\mathcal{S}_N]$ of $a = \sum_p a_p p$ is defined by $a^{\ast} = \sum_p a_p p^{-1}$.}
\end{itemize}
\end{Prop}
\begin{proof}
(i) and (ii) can be shown by easy calculations. To prove (iii) we calculate
\begin{equation} \label{e5}%
\langle au | v\rangle = \sum_p \sum_{\sigma} \sum_{\sigma'} a_p u_{\sigma} \overline{v}_{\sigma'} \langle\sigma\circ p^{-1} |\sigma'\rangle
\end{equation}
for $a = \sum_p a_p p\in\mathbb{C}[\mathcal{S}_N]$ and  $u = \sum_{\sigma} u_{\sigma} |\sigma\rangle$, $v = \sum_{\sigma'} v_{\sigma'} |\sigma'\rangle\in\mathcal{H}$. From
$\langle\sigma\circ p^{-1} |\sigma'\rangle = \delta_{\sigma\circ p^{-1} , \sigma'}$ we obtain
\begin{equation} \label{e6}%
\langle\sigma\circ p^{-1}  |\sigma'\rangle =
\langle\sigma |\sigma'\circ p\rangle =
\langle\sigma | p^{-1}|\sigma'\rangle\,.
\end{equation}
A substitution of (\ref{e6}) in (\ref{e5}) yields
\[
\langle au | v\rangle = \sum_{\sigma} \sum_{\sigma'} u_{\sigma} \overline{v}_{\sigma'} \bigg\langle\sigma \bigg| \sum_p\overline{a}_p p^{-1}\bigg|\sigma'\bigg\rangle =
\sum_{\sigma} \sum_{\sigma'} u_{\sigma} \overline{v}_{\sigma'} \langle\sigma | \overline{a}^{\ast}|\sigma'\rangle =
\langle u |\overline{a}^{\ast}v\rangle\,,
\]
which proves (iii).
\end{proof}

Let $N = r_1 + r_2$ be a decomposition of $N$ into two natural numbers. We denote by $\mathcal{B}^{(r_1 , r_2)}$ the subset
$\mathcal{B}^{(r_1 , r_2)} := \{ |\sigma\rangle\in\mathcal{B}\;|\;|{\sigma}^{-1}(k)| = r_k \ \mbox{for all} \ k\in\widehat{2}\}$ of all $|\sigma\rangle\in\mathcal{B}$ which contain $r_1$-times the $1$ and $r_2$-times the $2$. The 2-tuple $(r_1,r_2)$ is called the {\it weight} of the $|\sigma\rangle\in\mathcal{B}^{(r_1 , r_2)}$. Since $H_F$ and $H_A$ preserve $\mathcal{B}^{(r_1 , r_2)}$, the span
$\mathcal{H}^{(r_1 , r_2)} := \mathcal{L}_{\mathbb{C}}\mathcal{B}^{(r_1 , r_2)}\subset\mathcal{H}$ is an invariant subspace of $H_F$ and $H_A$.
\begin{Def}
If $G\subseteq\mathcal{S}_N$ is a subgroup of $\mathcal{S}_N$ then we denote by $1_G$ the group ring element $1_G := \frac{1}{|G|}\sum_{p\in G} p$ which is an idempotent. $|G|$ denotes the cardinality of $G$.
\end{Def}
The elements of a space $\mathcal{H}^{(r_1 , r_2)}$ can be generated by a suitable symmetry operator from a single basis vector from $\mathcal{B}^{(r_1 , r_2)}$. We denote by $\mathcal{S}_{r_1,r_2} \cong \mathcal{S}_{r_1}\times\mathcal{S}_{r_2}$ the Young subgroup of $\mathcal{S}_N$ the elements of which carry out permutations only within $1,\ldots,r_1$ and within $r_1 + 1,\ldots,N$. Let
$\mathfrak{R}_{r_1,r_2}$ be the set of those representatives of the left cosets $p\cdot\mathcal{S}_{r_1,r_2}$ of $\mathcal{S}_N$ relative to $\mathcal{S}_{r_1,r_2}$ which are the lexicographically smallest elements of their left coset. Further we use the idempotent
%\begin{equation}
$1_{\mathcal{S}_{r_1,r_2}} := \frac{1}{r_1!\,r_2!}\,\sum_{s\in\mathcal{S}_{r_1,r_2}}s$.
%\end{equation}
Now we have
\begin{Thm}
For every vector $w\in\mathcal{H}^{(r_1 , r_2)}$ there exists a unique
$a_w\in\mathcal{L}_{\mathbb{C}}\mathfrak{R}_{r_1,r_2}$ such that
\begin{equation} \label{e8}%
w = a_w\cdot 1_{\mathcal{S}_{r_1,r_2}} |{\tau}_0\rangle\;\;\;,\;\;\;
|{\tau}_0\rangle := |\underbrace{1,\ldots,1}_{r_1},\underbrace{2,\ldots,2}_{r_2}\rangle
\end{equation}
\end{Thm}
\begin{proof}
Obviously, for every $|\sigma\rangle\in\mathcal{B}^{(r_1,r_2)}$ there exists a $p\in\mathcal{S}_N$ such that $|\sigma\rangle = p|{\tau}_0\rangle$. $p$ belongs to exactly one left coset $r_{\sigma}\cdot\mathcal{S}_{r_1,r_2}$, where
$r_{\sigma}\in\mathfrak{R}_{r_1,r_2}$, i.e. $p = r_{\sigma}\circ s$ with $s\in\mathcal{S}_{r_1,r_2}$. Since
$1_{\mathcal{S}_{r_1,r_2}} |{\tau}_0\rangle = |{\tau}_0\rangle$ and
$s\cdot 1_{\mathcal{S}_{r_1,r_2}} = 1_{\mathcal{S}_{r_1,r_2}}$ for all
$s\in\mathcal{S}_{r_1,r_2}$, we obtain
$|\sigma\rangle = r_{\sigma}\cdot 1_{\mathcal{S}_{r_1,r_2}} |{\tau}_0\rangle$. Consequently, every
$w = \sum_{\sigma\in\mathcal{B}^{(r_1,r_2)}} w_{\sigma} |\sigma\rangle\in\mathcal{H}^{(r_1 , r_2)}$ fulfills (\ref{e8}) with
\begin{equation}
a_w := \sum_{\sigma\in\mathcal{B}^{(r_1,r_2)}} w_{\sigma} r_{\sigma}\in\mathcal{L}_{\mathbb{C}}\mathfrak{R}_{r_1,r_2}\,.
\end{equation}
$a_w$ is unique.
\end{proof}

Now we define symmetry classes in $\mathcal{H}$ in the same way as they can be defined for tensors.
\begin{Def}
Let $\mathcal{R}\subseteq\mathbb{C}[\mathcal{S}_N]$ be a right ideal of
$\mathbb{C}[\mathcal{S}_N]$. Then
\begin{equation}
\mathcal{H}_{\mathcal{R}} := \{ au\in\mathcal{H}\;|\;a\in\mathcal{R}\,,\,
u\in\mathcal{H}\}
\end{equation}
is called the symmetry class of $\mathcal{H}$ defined by $\mathcal{R}$.
\end{Def}
As for tensors one can prove (see e.g. \cite[Chap.~V, {\S}4]{boerner2} or \cite[p.~115]{fie16}).
\begin{Prop}
Let $e\in\mathbb{C}[\mathcal{S}_N]$ be a generating idempotent of a right ideal $\mathcal{R}\subseteq\mathbb{C}[\mathcal{S}_N]$, i.e.
$\mathcal{R} = e\cdot\mathbb{C}[\mathcal{S}_N]$. Then a $u\in\mathcal{H}$ is in $\mathcal{H}_{\mathcal{R}}$ iff $eu = u$.
\end{Prop}
\begin{Cor}
If $e, f$ are generating idempotents of a right ideal $\mathcal{R}\subseteq\mathbb{C}[\mathcal{S}_N]$, then $u\in\mathcal{H}_{\mathcal{R}}$ iff $eu = fu = u$.
\end{Cor}
\begin{Def}
We denote by $\mathcal{J}_0$ the set $\mathcal{J}_0 := \{a\in\mathbb{C}[\mathcal{S}_N]\;|\;au = 0 \forall\,u\in\mathcal{H}\}$.
\end{Def}
If $N > 2$ then $\mathcal{J}_0\not=\{0\}$, because then the idempotent
$e := \frac{1}{N!} \sum_p \mathrm{sign}(p) p$ satisfies $e|\sigma\rangle = 0$ for all $|\sigma\rangle\in\mathcal{B}$.
\begin{Prop}[{\cite[p.~116]{fie16}}]
$\mathcal{J}_0$ is a two-sided ideal of $\mathbb{C}[\mathcal{S}_N]$.
\end{Prop}
Every two-sided ideal of $\mathbb{C}[\mathcal{S}_N]$ has one and only one generating idempotent which is central. Let $f_0$ be the generating idempotent of $\mathcal{J}_0$. Then $f := id - f_0$ is also central idempotent which is orthogonal to $f_0$, i.e. $f\cdot f_0 = f_0\cdot f = 0$, and which generates a two-sided ideal $\mathcal{J} := f\cdot\mathbb{C}[\mathcal{S}_N]$ that fulfills $\mathbb{C}[\mathcal{S}_N] = \mathcal{J}\oplus \mathcal{J}_0$. $\mathcal{J}$ contains all those symmetry operators $a$ for which $\ker a\subset\mathcal{H}$. These are the symmetry operators in which we are interestet. (Compare \cite[p.~116]{fie16}.)

For a vector $w\in\mathcal{H}^{(r_1,r_2)}$ the smallest symmetry class which contains $w$ can be determined easily.
\begin{Thm} \label{thm3.13}%
Let $w\in\mathcal{H}^{(r_1,r_2)}$ be a fixed vector from $\mathcal{H}^{(r_1,r_2)}$ which we write in form {\rm (\ref{e8})}, i.e.
$w = a_w\cdot 1_{\mathcal{S}_{r_1,r_2}} |{\tau}_0\rangle$. Consider the right ideal
$\mathcal{R}_w := a_w\cdot 1_{\mathcal{S}_{r_1,r_2}}\cdot\mathbb{C}[\mathcal{S}_N]$.
\begin{itemize}
\item[\rm (i)]{Obviously, $w\in\mathcal{H}_{\mathcal{R}_w}$.}
\item[\rm (ii)]{Every right ideal $\mathcal{R}$ the symmetry class
$\mathcal{H}_{\mathcal{R}}$ of which contains $w$ satisfies
$\mathcal{R}_w\subseteq\mathcal{R}$.
}
\end{itemize}
\end{Thm}
\begin{proof}
Only (ii) requires a proof. A $w\in\mathcal{H}_{\mathcal{R}}$ can be written as $w = au$ with $a\in\mathcal{R}$ and $u\in\mathcal{H}$. We decompose $u$ into a linear combination $\tilde{u}$ of basis vectors from $\mathcal{B}^{(r_1,r_2)}$ and a linear combination $v$ of basis vectors from $\mathcal{B}\setminus\mathcal{B}^{(r_1,r_2)}$. Since
$w\in\mathcal{H}^{(r_1,r_2)}$ the parts $\tilde{u}$ and $v$ have to fulfill
$av = 0$ and $w = a\tilde{u}$.

Because $\tilde{u}\in\mathcal{H}^{(r_1,r_2)}$ the vector $\tilde{u}$ can be written in the form $\tilde{u} = a_{\tilde{u}}\cdot 1_{\mathcal{S}_{r_1,r_2}} |{\tau}_0\rangle$, too. Consequently, $w = a\cdot a_{\tilde{u}}\cdot 1_{\mathcal{S}_{r_1,r_2}} |{\tau}_0\rangle$.

Now we decompose the element $b := a\cdot a_{\tilde{u}}$ into parts corresponding to the left cosets of $\mathcal{S}_N$ relative to $\mathcal{S}_{r_1,r_2}$:
\begin{equation} \label{e12}%
b\;=\;\sum_{p\in\mathcal{S}_N} b_p p\;=\;
\sum_{r\in\mathfrak{R}_{r_1,r_2}}\sum_{p\in r\cdot\mathcal{S}_{r_1,r_2}} b_p p\;=\;
\sum_{r\in\mathfrak{R}_{r_1,r_2}} r\cdot\left(\sum_{s\in\mathcal{S}_{r_1,r_2}} b_{r\cdot s} s\right)\,.
\end{equation}
Since $s\cdot 1_{\mathcal{S}_{r_1,r_2}} = 1_{\mathcal{S}_{r_1,r_2}}$ for all
$s\in\mathcal{S}_{r_1,r_2}$ we obtain from (\ref{e12})
$b\cdot 1_{\mathcal{S}_{r_1,r_2}} = (\sum_{r\in\mathfrak{R}_{r_1,r_2}} B_r r)\cdot 1_{\mathcal{S}_{r_1,r_2}}$ with $B_r := \sum_{s\in\mathcal{S}_{r_1,r_2}} b_{r\cdot s}$. From this it follows
$w = a\cdot a_{\tilde{u}}\cdot 1_{\mathcal{S}_{r_1,r_2}} |{\tau}_0\rangle =
(\sum_{r\in\mathfrak{R}_{r_1,r_2}} B_r r)\cdot 1_{\mathcal{S}_{r_1,r_2}} |{\tau}_0\rangle$.
On the other hand, it holds $w = a_w\cdot 1_{\mathcal{S}_{r_1,r_2}} |{\tau}_0\rangle$ and $a_w\in\mathcal{L}_{\mathbb{C}}\mathfrak{R}_{r_1,r_2}$ is unique. This leads to $a_w = \sum_{r\in\mathfrak{R}_{r_1,r_2}} B_r r$ and
$a\cdot a_{\tilde{u}}\cdot 1_{\mathcal{S}_{r_1,r_2}} = a_w\cdot 1_{\mathcal{S}_{r_1,r_2}}$. The last relation yields $a_w\cdot 1_{\mathcal{S}_{r_1,r_2}}\in\mathcal{R}$ and $\mathcal{R}_w\subseteq\mathcal{R}$.
\end{proof}
\begin{Thm}
Let $U\subseteq\mathcal{H}^{(r_1,r_2)}$ be a linear subspace of $\mathcal{H}^{(r_1,r_2)}$ with a basis $\{v_1,...,v_l\}$. Consider the right ideal
$\mathcal{R}_U := \sum_{k=1}^l \mathcal{R}_{v_k}$ (non-direct sum).
\begin{itemize}
\item[\rm (i)]{Obviously, $U\subseteq\mathcal{H}_{\mathcal{R}_U}$.}
\item[\rm (ii)]{Every right ideal $\mathcal{R}$ the symmetry class
$\mathcal{H}_{\mathcal{R}}$ of which contains $U$ satisfies
$\mathcal{R}_U\subseteq\mathcal{R}$.
}
\end{itemize}
\end{Thm}
\begin{proof}
Ad (i): An arbitrary $u\in U$ can be written as
$u = \sum_k x_k v_k$. If we express every $v_k$ in form (\ref{e8}), i.e.
$v_k = a_{v_k}\cdot 1_{\mathcal{S}_{r_1,r_2}} |{\tau}_0\rangle$, then we obtain
$u = (\sum_k x_k a_{v_k}\cdot 1_{\mathcal{S}_{r_1,r_2}}) |{\tau}_0\rangle$. This leads to $u\in\mathcal{H}_{\mathcal{R}_U}$, because
$\sum_k x_k a_{v_k}\cdot 1_{\mathcal{S}_{r_1,r_2}}\in\sum_k \mathcal{R}_{v_k} = \mathcal{R}_U$.

Ad (ii): If $\mathcal{H}_{\mathcal{R}}$ is a symmetry class with
$U\subseteq\mathcal{H}_{\mathcal{R}}$, then every $v_k$ lies in $\mathcal{H}_{\mathcal{R}}$. From this Theorem \ref{thm3.13} yields
$\mathcal{R}_{v_k}\subseteq\mathcal{R}$ for every $k$. Consequently,
$\mathcal{R}_U = \sum_k \mathcal{R}_{v_k}\subseteq\mathcal{R}$, too.
\end{proof}
\begin{Cor}
If two bases $\{v_1,...,v_l\}$ and $\{v'_1,...,v'_l\}$ of $U$ are given and we form $\mathcal{R}_U := \sum_{k=1}^l \mathcal{R}_{v_k}$ and
$\mathcal{R}'_U := \sum_{k=1}^l \mathcal{R}_{v'_k}$ then
$\mathcal{R}_U = \mathcal{R}'_U$.
\end{Cor}
A generating idempotent $e$ of $\mathcal{R}_U$ and a decomposition $e = e_1 +\ldots + e_m$ into pairwise orthogonal primitive idempotents $e_k$ can be determined by means of the decomposition algorithm from \cite[Chap.~I]{fie16} or \cite{fie18} which is implemented in the {\sf Mathematica} package {\sf PERMS} \cite{fie10}.

For tensors the following types of symmetries can be defined:
\begin{itemize}
\item[I.]{Commutation symmetries.}
\item[II.]{Symmetries defined by irreducible characters of subgroups $G\subseteq\mathcal{S}_N$.}
\item[III.]{Symmetry classes.}
\item[IV.]{Symmetries defined by a finite set of anihilating symmetry operators.}
\end{itemize}
See \cite[pp.~114]{fie16} for details. All these symmetry types can be defined in $\mathcal{H}$, too. In the present paper we consider III. and I.
\begin{Def}
Let $C\subseteq\mathcal{S}_N$ be a subgroup of $\mathcal{S}_N$ and
$\epsilon: C\rightarrow\mathcal{S}^1$ be a homomorphism of $C$ onto a finite subgroup in the group
$\mathcal{S}^1 = \{z\in\mathbb{C}\;|\;|z| = 1\}$ of complex units. We say that $u\in\mathcal{H}$ possesses the commutation symmetry $(C,\epsilon)$ if
$cu = \epsilon(c)u$ for all $c\in C$.
\end{Def}
\begin{Prop}[{\cite[p.~115]{fie16}}] \label{prop2.14}%
Let $(C,\epsilon)$ be a commutation symmetry.
\begin{itemize}
\item[\rm (i)]{Then the group ring element
%\begin{equation}
$\epsilon := \frac{1}{|C|}\,\sum_{c\in C}\epsilon(c)c$
%\end{equation}
is an idempotent of $\mathbb{C}[C]\subseteq\mathbb{C}[\mathcal{S}_N]$.}
\item[\rm (ii)]{A $u\in\mathcal{H}$ has the symmetry $(C,\epsilon)$ iff
${\epsilon}^{\ast}u = u$.}
\end{itemize}
\end{Prop}
A list of all commutation symmetries belonging to subgroups of $\mathcal{S}_N$ with $N\le 6$ is given in \cite[Appendix A.1]{fie16}.

\section{Symmetry operators which commute with restrictions of the Hamiltonian} \label{sec4}%
It is very importand to find symmetry operators which commute with the Hamiltonian $H$ (in a certain sense).
\begin{Def} \label{def4.1}
Let $H = H_F , H_A$. Let $v\in\mathcal{H}^{(r_1,r_2)}$ be an eigenvector of $H$ and let $U\subseteq\mathcal{H}^{(r_1,r_2)}$ be a linear subspace spanned by eigenvectors of a fixed eigenvalue $\mu$ of $H$.
\begin{itemize}
\item[\rm (i)]{
We denote by $\mathcal{A}_v$ the stability subgroup
$\mathcal{A}_v := \{p\in\mathcal{S}_N\;|\;pv = v\}$.}
\item[\rm (ii)]{
We denote by $\mathcal{A}_U$ the stability subgroup
$\mathcal{A}_U := \{p\in\mathcal{S}_N\;|\;pU\subseteq U\}$.}
\end{itemize}
\end{Def}
The following statements can be proved by easy considerations:
\begin{Thm}
If $v$ is a fixed eigenvector of $H$ then
$\mathcal{A}_{\mathcal{L}_{\mathbb{C}}\{v\}}$ is the maximal subgroup of $\mathcal{S}_N$ which yields a commutaion symmetry $(\mathcal{A}_{\mathcal{L}_{\mathbb{C}}\{v\}},\epsilon)$ for $v$. Then $\ker\epsilon = \mathcal{A}_v$.
\end{Thm}
\begin{Thm}
Every $p\in\mathcal{A}_{\mathcal{L}_{\mathbb{C}}\{v\}}$ or $p\in\mathcal{A}_U$ commutes with $H|_{\mathcal{L}_{\mathbb{C}}\{v\}}$ or $H|_U$, respectively.
\end{Thm}
\begin{Thm}
Let $\mu$ be an eigenvalue of the Hamiltonian $H$ and denote by $U_{\mu}$ the eigenspace belonging to $\mu$. Consider $U_{\mu}^{(r_1,r_2)} := U_{\mu}\cap\mathcal{H}^{(r_1,r_2)}$. Let $e$ be a generating idempotent of the right ideal
$\mathcal{R}_{U_{\mu}^{(r_1,r_2)}}$. Then it holds
$H|_{U_{\mu}^{(r_1,r_2)}}\cdot e = e\cdot H|_{U_{\mu}^{(r_1,r_2)}}$.
\end{Thm}
\section{Eigenvectors with reduced symmetry classes} \label{sec4a}%
When we determine $\mathcal{R}_v$ and $\mathcal{H}_{\mathcal{R}_v}$ for all eigenvectors $v$ of $H$ from a space $U_{\mu}^{(r_1,r_2)}$ then the dimensions
$\dim\mathcal{R}_v$ and $\dim\mathcal{H}_{\mathcal{R}_v}$ are constant almost everywhere on $U_{\mu}^{(r_1,r_2)}$. However, it is possible that $U_{\mu}^{(r_1,r_2)}$ contains certain linear subspaces on which these dimensions jump to smaller values. We describe, how one can determine such jumps.

Let $\{v_1,\ldots,v_k\}\subset\mathcal{H}^{(r_1,r_2)}$ be a set of linearly independent eigenvectors of a fixed eigenvalue $\mu$ of $H = H_F, H_A$ and
$U := \mathcal{L}_{\mathbb{C}}\{v_1,\ldots,v_k\}$. We consider the family
\begin{equation}
v = v(x_1,\ldots,x_k) = \sum_{l=1}^k x_l\cdot v_l\;\;\;,\;\;\;x_1,\ldots,x_k\in\mathbb{C}
\end{equation}
and investigate for every $v$ the right ideal $\mathcal{R}_v$ according to Theorem \ref{thm3.13} which defines the smallest symmetry class containing $v$. We express every $v_l$ in the form (\ref{e8}), i.e.
$v_l = a_{v_l}\cdot 1_{\mathcal{S}_{r_1,r_2}} |{\tau}_0\rangle$. The idempotent $1_{\mathcal{S}_{r_1,r_2}}$ generates a left ideal
$\mathbb{C}[\mathcal{S}_N]\cdot 1_{\mathcal{S}_{r_1,r_2}}$ which is the representation space of the Littlewood-Richardson product $[r_1][r_2]$. Since
$[r_1][r_2] = [r_2][r_1]$
we may assume $r_1\ge r_2$ w.l.o.g. Then the Littlewood-Richardson rule yields
\begin{equation}
[r_1][r_2]\sim [N] + [N-1,1] + [N-2,2] + \ldots + [N-r_2,r_2]\,.
\end{equation}
Consequently, $1_{\mathcal{S}_{r_1,r_2}}$ possesses a decomposition
\begin{equation}
1_{\mathcal{S}_{r_1,r_2}} = e_{(N)} + e_{(N-1,1)} + e_{(N-2,2)} +\ldots + e_{(N-r_2,r_2)}
\end{equation}
into pairwise orthogonal, primitive idempotents $e_{(N-m,m)}$ belonging to the partitions $(N-m,m)\vdash N$. The idempotents $e_{(N-m,m)}$ can be calculated by the formula
\begin{equation}
e_{(N-m,m)} = D^{-1}\left(
\begin{array}{ccc}
0 & & \\
 & D_{(N-m,m)}(1_{\mathcal{S}_{r_1,r_2}}) & \\
 & & 0 \\
\end{array}
\right)
\end{equation} where $D$ is a discrete Fourier transform of $\mathcal{S}_N$ and $D_{\lambda}$ the natural projection of $D$ belonging to the partition $\lambda\vdash N$ (see \cite{clausbaum1} or \cite[p.~27]{fie16}). We use {\it Young's natural representation} \cite[p.~51]{fie16} as discrete Fourier transform which is implemented in our {\sf Mathematica} package {\sf PERMS} \cite{fie10}.

Because of Theorem \ref{thm3.13} the smallest symmetry class which contains $v$ is defined by the right ideal
\begin{equation} \label{e17}%
\mathcal{R}_v = \bigoplus_{m = 0}^{r_2} \left(\sum_{l=1}^k x_l\,a_{v_l}\right)\cdot e_{(N-m,m)}\cdot\mathbb{C}[\mathcal{S}_N]\,.
\end{equation}
The outer sum is direct since every summand
$\left(\sum_{l=1}^k x_l\,a_{v_l}\right)\cdot e_{(N-m,m)}\cdot\mathbb{C}[\mathcal{S}_N]$ is a (minimal) subideal of a minimal two-sided ideal $J_{(N-m,m)}$ occuring in the decomposition
$\mathbb{C}[\mathcal{S}_N] = \bigoplus_{\lambda\vdash N} J_{\lambda}$ of the group ring $\mathbb{C}[\mathcal{S}_N]$ into minimal two-sided ideals $J_{\lambda}$.

Now we investigate for every $e_{(N-m,m)}$ whether the linear equation system
\begin{equation} \label{e18}%
\sum_{l=1}^k x_l\,a_{v_l}\cdot e_{(N-m,m)} = 0
\end{equation}
has a solution $\tilde{x} = (\tilde{x}_1,\ldots,\tilde{x}_k)\not= 0$. If ''yes'', then a summand in (\ref{e17}) vanishes for $\tilde{x}$ and the symmetry class $\mathcal{H}_{\mathcal{R}_{\tilde{v}}}$ of $\tilde{v}$ belonging to $\tilde{x}$ is smaller the symmetry class of
$\mathcal{H}_{\mathcal{R}_v}$ in the generic case in which no summand is missing in (\ref{e17}).

If (\ref{e18}) has non-vanishing solutions for $m$ and $m'$, then we can investigate whether the union of the two systems
\begin{equation}
\sum_{l=1}^k x_l\,a_{v_l}\cdot e_{(N-m,m)} = 0\;\;\;,\;\;\;
\sum_{l=1}^k x_l\,a_{v_l}\cdot e_{(N-m',m')} = 0
\end{equation}
possesses still a non-vanishing solution. Such a solution cancels two summands in (\ref{e17}). A continuation of this process can yield a further reduction of (\ref{e17}).

Note that (\ref{e18}) leads to a system of $N!$ linear equations for $x_1,\ldots,x_k$. However, a discrete Fourier transform $D$ for $\mathcal{S}_N$ transforms (\ref{e18}) into smaller, equivalent system of $n_{\lambda}^2$ linear equations for $x_1,\ldots,x_k$, where $n_{\lambda}$ can be calculated from $\lambda = (N-m,m)$ by means of the hook length formula (see \cite[p.~38]{fie16}). One obtains $n_{(N-m,m)} = \frac{N!}{(N-m+1)!\,m!}\cdot (N-2m+1)$.

\section{Examples in the case $N = 4$} \label{sec5}%
Now we give some examples in the case $N = 4$ which we calculated by means of the {\sf Mathematica} packages {\sf PERMS} \cite{fie10} and {\sf HRing} \cite{fie07a}.
Instead of $H_F$ and $H_A$ we consider
\begin{equation}
\tilde{H} = -\textstyle{\frac{1}{J}} H_F = \textstyle{\frac{1}{J}} H_A.
\end{equation}
We calculate the eigenvalues $\mu$ and eigenvectors $v$ of $\tilde{H}$ by numerical methods (see Sec.~\ref{sec2}) and determine then stability subgroups
$\mathcal{A}_v$ and $\mathcal{A}_{\mathcal{L}_{\mathbb{C}}\{v\}}$ according to Def.~\ref{def4.1} by means of an algorithm from \cite[p.~28]{butler} which is implemented in our {\sf Mathematica}-package {\sf PERMS} \cite{fie10}. The following table shows the results.
\begin{center}
{\small
\begin{tabular}{r|c|lc|c|c}
$\mu$ & $(r_1,r_2)$ & & eigenvector $v$ & $|\mathcal{A}_v|$ & $|\mathcal{A}_{\mathcal{L}_{\mathbb{C}}\{v\}}|$ \\
\hline
-2 & (2,2) & $v_1$: & $|1 1 2 2 \rangle - 2 |1 2 1 2 \rangle + |1 2 2 1 \rangle + |2 1 1 2 \rangle - 2 |2 1 2 1 \rangle + |2 2 1 1 \rangle$ & 8 & $8$\\
\hline
-1 & (1,3) & $v_2$: & $-|1 2 2 2 \rangle + |2 1 2 2 \rangle - |2 2 1 2 \rangle + |2 2 2 1 \rangle$ & 4 & 8\\
-1 & (2,2) & $v_3$: & $-|1 2 1 2 \rangle + |2 1 2 1 \rangle$ & 4 & 8\\
-1 & (3,1) & $v_4$: & $-|1 1 1 2 \rangle + |1 1 2 1 \rangle - |1 2 1 1 \rangle + |2 1 1 1 \rangle$ & 4 & 8\\
\hline
0 & (1,3) & $v_5$: & $-|2 1 2 2 \rangle + |2 2 2 1 \rangle$ & 2 & 4\\
0 & (1,3) & $v_6$: & $-|1 2 2 2 \rangle + |2 2 1 2 \rangle$ & 2 & 4\\
0 & (2,2) & $v_7$: & $-|1 1 2 2 \rangle + |2 2 1 1 \rangle$ & 4 & $8$\\
0 & (2,2) & $v_8$: & $-|1 1 2 2 \rangle + |2 1 1 2 \rangle$ & 1 & $2$\\
0 & (3,1) & $v_9$: & $-|1 1 2 1 \rangle + |2 1 1 1 \rangle$ & 2 & 4\\
0 & (2,2) & $v_{10}$: & $-|1 1 2 2 \rangle + |1 2 2 1 \rangle$ & 1 & $2$\\
0 & (3,1) & $v_{11}$: & $-|1 1 1 2 \rangle + |1 2 1 1 \rangle$ & 2 & 4\\
\hline
1 & (0,4) & $v_{12}$: & $|2 2 2 2 \rangle$ & 24 & 24\\
1 & (1,3) & $v_{13}$: & $|1 2 2 2 \rangle + |2 1 2 2 \rangle + |2 2 1 2 \rangle + |2 2 2 1 \rangle$ & 24 & 24\\
1 & (2,2) & $v_{14}$: & $|1 1 2 2 \rangle + |1 2 1 2 \rangle + |1 2 2 1 \rangle + |2 1 1 2 \rangle + |2 1 2 1 \rangle + |2 2 1 1 \rangle$ & 24 & 24\\
1 & (3,1) & $v_{15}$: & $|1 1 1 2 \rangle + |1 1 2 1 \rangle + |1 2 1 1 \rangle + |2 1 1 1 \rangle$ & 24 & 24\\
1 & (4,0) & $v_{16}$: & $|1 1 1 1 \rangle$ & 24 & 24\\
\hline
\end{tabular}
}
\end{center}
We read from the table:
\begin{itemize}
\item[a)] All eigenvectors $v_{12},\ldots,v_{16}$ of $\mu = 1$ are fixed points of every $p\in\mathcal{S}_4$. Consequently, the complete eigenspace $U_0$ of $\mu = 0$ consists of fixed points of $\mathcal{S}_4$.
\item[b)] There are eigenvectors which are fixed points only of
$\mathrm{id} = (1 2 3 4)\in\mathcal{S}_4$ (see $v_8$, $v_{10}$).
\end{itemize}
The permutation $t := (4 1 2 3)$ represents the {\it translation operator}
$|\sigma\rangle\mapsto t |\sigma\rangle = |\sigma\circ t^{-1}\rangle$ on $\mathcal{H}$. It generates the cyclic group $C_4 = \{(1 2 3 4), (2 3 4 1), (3 4 1 2), (4 1 2 3)\}$. But only the groups $\mathcal{A}_{\mathcal{L}_{\mathbb{C}}\{v\}}$ of $v_1$ and $v_{12},\ldots,v_{16}$ contain $C_4$, i.e. many single eigenvectors have no translational symmetry. Even the group $\mathcal{A}_{\mathcal{L}_{\mathbb{C}}\{v_7\}}$ of order 8 does not contain $C_4$.
\begin{center}
{\small
\begin{tabular}{c|c}
$v$ & $\mathcal{A}_{\mathcal{L}_{\mathbb{C}}\{v\}}$ \\
\hline
$v_1$ & $\{( 1 2 3 4 ), ( 1 4 3 2 ), ( 2 1 4 3 ), ( 4 1 2 3 ), ( 2 3 4 1 ), ( 4 3 2 1 ), ( 3 2 1 4 ), ( 3 4 1 2 )\}$ \\
\hline
$v_7$ & $\{( 1 2 3 4 ), ( 1 2 4 3 ), ( 2 1 3 4 ), ( 3 4 1 2 ), ( 2 1 4 3 ), ( 4 3 1 2 ), ( 3 4 2 1 ), ( 4 3 2 1 )\}$ \\
$v_8$ & $\{( 1 2 3 4 ), ( 3 2 1 4 )\}$ \\
$v_{10}$ & $\{( 1 2 3 4 ), ( 1 4 3 2 )\}$ \\
\end{tabular}
}
\end{center}
However, it holds:
\begin{Prop}
For every space $U_{\mu}^{(r_1,r_2)}$ the stability subgroup
$\mathcal{A}_{U_{\mu}^{(r_1,r_2)}}$ contains $C_4$. 
\end{Prop}
\begin{proof}
For $\mu = 1$ the statement is obviously correct.
Further we can read from our table of eigenvectors that $tU\subseteq U$ for
$U = U_{-2}^{(2,2)}, U_{-1}^{(1,3)}, U_{-1}^{(2,2)}, U_{-1}^{(3,1)}, U_{0}^{(1,3)}, U_{0}^{(3,1)}$. Finally, a computer calculation yields
$\mathcal{A}_{U_0^{(2,2)}} = \{
( 1 2 3 4 ), ( 1 4 3 2 ), ( 2 1 4 3 ), ( 4 1 2 3 ), ( 2 3 4 1 ), ( 4 3 2 1 ), ( 3 2 1 4 ), ( 3 4 1 2 )\}$ for $U_0^{(2,2)} = \mathcal{L}_{\mathbb{C}}\{v_7, v_8, v_{10}\}$.
\end{proof}

Now we consider single eigenvectors $v$ of weight $(2,2)$ and investigate the right ideals $\mathcal{R}_v$ of their smallest symmetry class and the right ideal
$1_{\mathcal{A}_v}\cdot\mathbb{C}[\mathcal{S}_4]$ of the symmetry class given by the commutation symmetry
$(\mathcal{A}_v,1)$ (see Prop.~\ref{prop2.14}). The next table shows the structure of their decompositions into minimal right ideals. Clearly, the ideals $\mathcal{R}_v$ are ''much smaller'' then the ideals $1_{\mathcal{A}_v}\cdot\mathbb{C}[\mathcal{S}_4]$.
\begin{center}
{\small
\begin{tabular}{c|c|c|c}
$\mu$ & $v$  & $\mathcal{R}_v$ & $1_{\mathcal{A}_v}\cdot\mathbb{C}[\mathcal{S}_4]$ \\
\hline
$-2$ & $v_1$ & $[2\,2]$ & $[4] + [2\,2]$\\
\hline
$-1$ & $v_3$ & $[3\,1]$ & $[4] + [3\,1] + [2\,2]$\\
\hline
$0$ & $v_7$ & $[3\,1]$ & $[4] + [3\,1] + [2\,2]$\\
$0$ & $v_8$ & $[3\,1] + [2\,2]$ & $\mathbb{C}[\mathcal{S}_4]\sim [4] + 3 [3\,1] + 2 [2\,2] + 3 [2\,1\,1\,1] + [1\,1\,1\,1]$ \\
$0$ & $v_{10}$ & $[3\,1] + [2\,2]$ & $\mathbb{C}[\mathcal{S}_4]\sim [4] + 3 [3\,1] + 2 [2\,2] + 3 [2\,1\,1\,1] + [1\,1\,1\,1]$ \\
\hline
$1$ & $v_{14}$ & $[4]$ & $[4]$ \\
\end{tabular}
}
\end{center}
Very interesting is the space $U_0^{(2,2)} := \mathcal{L}_{\mathbb{C}}\{v_7, v_8, v_{10}\}$. Using the decomposition algorithm from \cite[Chap.~I]{fie16} or \cite{fie18} we decompose
$\mathcal{R}_{U_0^{(2,2)}} = \mathcal{R}_{v_7} + \mathcal{R}_{v_8} + \mathcal{R}_{v_{10}}$ into minimal right ideals.
\begin{Prop}
For $U_0^{(2,2)} := \mathcal{L}_{\mathbb{C}}\{v_7, v_8, v_{10}\}$ we have
$\mathcal{R}_{U_0^{(2,2)}}\sim 2 [3\,1] + [2\,2]$.
\end{Prop}
Clearly, the ideals $\mathcal{R}_u$ of single (generic) eigenvectors $u\in U_0^{(2,2)}$ are ''smaller'' than $\mathcal{R}_{U_0^{(2,2)}}$. Moreover, $U_0^{(2,2)}$ contains two linear subspaces over which a reduction of $\mathcal{R}_u$ arises (see Sec.~\ref{sec4a}).
\begin{center}
{\small
\begin{tabular}{l|c}
 & $\mathcal{R}_u$ \\
\hline
$u\in\mathcal{L}_{\mathbb{C}}\{v_7,v_8,v_{10}\}$ generic & $[3\,1] + [2\,2]$\\
\hline
$u\in\mathcal{L}_{\mathbb{C}}\{v_7 - v_8 - v_{10}\}$ & $[2\,2]$ \\
$u\in\mathcal{L}_{\mathbb{C}}\{v_7,v_8 - v_{10}\}$ & $[3\,1]$ \\
\end{tabular}
}
\end{center}
Note that $v_7$ reduces $\mathcal{R}_v$, since $v_7\in\mathcal{L}_{\mathbb{C}}\{v_7,v_8 - v_{10}\}$. This coincides with the fact that $\mathcal{A}_{\mathcal{L}_{\mathbb{C}}\{v\}}$ is much bigger for $v_7$ than for $v_8,v_{10}$.
Explicite date about all groups and right ideals discussed in Sec.~\ref{sec5} can be found online in the {\sf Mathematica} notebooks \cite[symgroups\_\hspace*{2pt}n4.nb, symclass\_\hspace*{2pt}n4.nb]{fie21}.

\section*{References}
%\bibliography{bflit}
%%%%%%%\bibliographystyle{plain} nicht fuer SSPCM
%\bibliographystyle{unsrt}
%\end{document}

\end{document}